\documentclass[12pt,oneside,reqno]{amsart}
\usepackage{graphicx}
\usepackage{indentfirst,csquotes}

\usepackage[margin=3cm]{geometry}  
\usepackage{amssymb,amsthm,amsmath}
\usepackage{xcolor,paralist,hyperref,
fancyhdr,etoolbox}
\usepackage{bm}

\newtheorem{theorem}{Theorem}[section]

\newtheorem{lemma}[theorem]{Lemma}
\newtheorem{proposition}[theorem]{Proposition}

\theoremstyle{definition} 
\newtheorem{example}[theorem]{Example}

\numberwithin{equation}{section} 

\hypersetup{ colorlinks=true, linkcolor=black, filecolor=black, urlcolor=black }

\usepackage{mathrsfs}
\usepackage{mathtools}
\usepackage{mathdots}
\usepackage{stmaryrd}
\usepackage[all,cmtip]{xy}
\usepackage{stackengine} 
\newcommand{\Mod}[1]{\ (\mathrm{mod}\ #1)}
\usepackage{stackengine} 
\usepackage{tikz-cd}

\def\N{\operatorname{\mathbb{N}}}

\def\Q{\operatorname{\mathbb{Q}}}
\def\R{\operatorname{\mathbb{R}}}
\def\height{\operatorname{height}}
\def\fpt{\operatorname{fpt}}
\def\lct{\operatorname{lct}}
\def\pf{\operatorname{Pf}}
\def\cl{\operatorname{Cl}}

\begin{document}
\title{$F$-pure Threshold for the Symmetric Determinantal Ring} 
\author{Justin Fong}
\date{}
\address{Department of Mathematics, Purdue University, West Lafayette, IN
47907, USA}
\email{jafong1@gmail.com}

\let\thefootnote\relax

\begin{abstract}
We give a value for the $F$-pure threshold at the maximal homogeneous ideal $\mathfrak{m}$ of the symmetric determinantal ring over a field of prime characteristic. The answer is characteristic independent, so we immediately get the log canonical threshold in characteristic zero as well. 
\end{abstract} 

\maketitle

\section{Introduction}

The $F$-pure threshold, introduced by Takagi and Watanabe in \cite{MR2097584},
is a numerical invariant associated to rings with $F$-pure singularities. $F$-pure singularities are an example class of $F$-singularities, which are classes of singularities defined in positive characteristic using the Frobenius map. Suppose $R$ is a reduced $F$-finite $F$-pure ring of characteristic $p>0$, and let $\mathfrak{a}\subseteq R$ be a nonzero ideal. For a real number $t\in\R_{\ge0}$, the pair $(R,\mathfrak{a}^t)$ is called \emph{$F$-pure} if for every $q=p^e\gg0$, there exists an element $f\in \mathfrak{a}^{\lfloor (q-1)t \rfloor}$ (where $\lfloor (q-1)t \rfloor$ is the greatest integer less than or equal to the real number $(q-1)t$) such that the map $R\to R^{1/q}$ which sends 1 to $f^{1/q}$ splits as an $R$-module homomorphism. The \emph{$F$-pure threshold} of $\mathfrak{a}$ is defined to be
\[\fpt(\mathfrak{a}) := \sup\{t\in\R_{\ge0} \mid (R,\mathfrak{a}^t) \ \text{is $F$-pure}\}.\]
When $\mathfrak{a}=\mathfrak{m}$ is the maximal homogeneous ideal of $R$, we simply refer to $\fpt(\mathfrak{m})$ as the $F$-pure threshold of $R$. 

The $F$-pure threshold has been calculated for the determinantal rings of generic matrices \cite[Proposition 4.3]{MR3719471} and Hankel matrices \cite[Theorem 4.5]{MR3836659}. As for skew-symmetric matrices, the $F$-pure threshold of its determinantal ring, commonly called the Pfaffian ring, is easy to calculate: It is equal to the negative of its $a$-invariant by \cite[Theorem 5.2]{MR3778235}, since these rings are always Gorenstein \cite{MR0570242}. To state this explicitly, let $X=(x_{ij})$ be an $n\times n$ skew-symmetric matrix, $\Bbbk[X]$ be the polynomial ring in the entries of $X$, and $\pf_{2t}(X)$ be the ideal generated by the $2t$-Pfaffians of $X$. The Pfaffian ring is $R=\Bbbk[X]/\pf_{2t}(X)$, which is known to be strongly $F$-regular \cite{MR1888527}. The $a$-invariant of $R$ has been computed in \cite[Corollary 1.7]{MR1188581}. Assuming $R$ has the standard grading, one has $a(R)=-n(t-1)$ if $2t\le n$, while $R=\Bbbk[X]$ if $2t>n$ (see \cite[Corollary 4.3]{MR1902527}). 

To the author's knowledge, $\fpt(\mathfrak{m})$ has been unknown for the symmetric matrix case, where the calculations are more involved since symmetric determinantal rings are not always Gorenstein. We provide an answer in Theorem \ref{thm: fpt symm det ring}, as well as a proof that the symmetric determinantal ring is $F$-pure, a fact not clearly stated in the literature.

\subsection*{Note} 
At the time of writing this article, the author was notified of the recent paper \cite{CSV}, which includes a proof of the more general property of $F$-regularity for symmetric determinantal rings, and mentions the same conclusion of Proposition \ref{prop: second symbolic power} in this article. Our work is independent of that paper.

\subsection*{Acknowledgements}
I would like to thank Irena Swanson, Vaibhav Pandey, and my advisor Uli Walther for providing helpful discussions. Partial support by NSF Grant DMS-2100288 is gratefully acknowledged. 

\section{Symmetric Determinantial Rings}

Suppose that $X=(x_{ij})$ is a symmetric $n\times n$ generic matrix (i.e., $x_{ij}=x_{ji}$), let $\Bbbk[X]=\Bbbk[x_{ij}\mid1\le i\le j\le n]$ be the polynomial ring over a field $\Bbbk$ with the entries of $X$ as its variables, and $I_t(X)$ be the ideal generated by the $t$-minors of $X$. For an integer $1\le t\le n$ let $R=\Bbbk[X]/I_t(X)$ be the \emph{symmetric determinantal ring}. A formula for the $a$-invariant of $R$ has been given by Barile in \cite{MR1255876}. When $R$ is given the standard grading this is given as follows:
\begin{equation}\label{equ: symm det a-invariant}
    a(R) = \begin{cases}
              -(t-1)\frac{n}{2}, & \text{if $n\equiv t \Mod 2$} \\
              -(t-1)\frac{n+1}{2}, & \text{if $n\not\equiv t \Mod 2$}.
            \end{cases}
\end{equation}
In \cite{MR1309380} it is shown that the graded canonical module of $R$ is
\begin{equation}\label{equ: symm det graded canonical mod}
    \omega_R = \begin{cases}
              R\left(-(t-1)\frac{n}{2}\right), & \text{if $n\equiv t \Mod 2$} \\
              \mathfrak{p}\left(-(t-1)\frac{n-1}{2}\right), & \text{if $n\not\equiv t \Mod 2$},
               \end{cases}
\end{equation}
where $\mathfrak{p}$ is the prime ideal generated by the $t-1$ minors of the first $t-1$ rows of $X$, which generates the divisor class group of $R$ (see \cite{MR0435063}). While $R$ is Gorenstein follows from (\ref{equ: symm det graded canonical mod}), this was first proven by Goto in \cite{MR0545715}. 

\begin{lemma}\cite[(4.2)]{MR4603914}\label{lem: sym power in colon ideal}
Suppose $S$ is a polynomial ring over a field of characteristic $p>0$, and let $I\subset S$ be a prime ideal of $h=\height(I)$. Then $I^{(h(p-1))}\subseteq(I^{[p]}:_SI)$, i.e., $I^{(h(p-1))+1}\subseteq I^{[p]}$.
\end{lemma}

\begin{proof}
To show the containment $I^{(h(p-1))+1}\subseteq I^{[p]}$, it is enough to show it locally at all the associated primes of $S/I^{[p]}$. The Frobenius map on $S$ is flat, the set of associated primes of $S/I^{[p]}$ is the same as for $S/I$, which consists only of $I$ since it is prime. In the regular local ring $(S_I,IS_I)$ the symbolic power of $IS_I$ coincides with its ordinary power since $IS_I$ is a maximal ideal. Hence it follows that
\[(IS_I)^{(h(p-1))}I=(IS_I)^{h(p-1)}I=(IS_I)^{h(p-1)+1}.\]
Since $(S_I,IS_I)$ has dimension $h=\height(I)$, its unique maximal ideal $IS_I$ is generated by $h$ elements $g_1,\dots,g_h$. The power $(IS_I)^{h(p-1)+1}$ is generated by $g_1^{a_i}\cdots g_h^{a_h}$ such that $a_1+\cdots+a_h=h(p-1)+1$. By the pigeonhole principle, at least one $a_i\ge p$, hence it follows that $(IS_I)^{h(p-1)+1}\subseteq I^{[p]}S_I$.
\end{proof}

\begin{proposition}
Over a field of characteristic $p>0$ the symmetric determinantal ring $R$ is $F$-pure.
\end{proposition}

\begin{proof}
Let $I_t=I_t(X)$, and $S=\Bbbk[X]$ be the polynomial ring in $n(n+1)/2$ many variables. We apply Fedder's criterion \cite{MR701505}: $R$ is $F$-pure if and only if $(I_t^{[p]}:_SI_t)\nsubseteq\mathfrak{m}_S^{[p]}$, where $\mathfrak{m}_S=(x_{ij})$ is the maximal homogeneous ideal of $S$. Thus, in order to show that $R$ is $F$-pure, we need to find an element $f\in (I_t^{[p]}:_SI_t)\setminus\mathfrak{m}_S^{[p]}$. 

Let $\mathfrak{D}=d_1\cdots d_n$ be a product of $i$-minors $d_i$ of the $i\times i$ submatrices in the bottom-left corner of $X$ (for $i=1,\dots,n$), whose diagonals are the $n$ distinct diagonals of $X$ (see Example \ref{ex: 3x3 symm matrix} below). Our proposed element is $f=\mathfrak{D}^{p-1}$. The leading term of $\mathfrak{D}$ with respect to the lexicographical ordering
\[x_{11}>x_{12}>\cdots>x_{1n}>x_{21}>\cdots>x_{n-1n}>x_{nn}\]
is $\mathrm{in}(\mathfrak{D})=\prod_{1\le i\le j\le n}x_{ij}$, which is the product of all the distinct variables of $X$, hence it is a squarefree monomial. This implies $\mathfrak{D}^{p-1}\notin\mathfrak{m}_S^{[p]}$.

Next, we need to show $\mathfrak{D}^{p-1}I_t\subseteq I_t^{[p]}$, i.e., $\mathfrak{D}^{p-1}\in(I_t^{[p]}:_SI_t)$. It suffices to prove this containment after localizing at the prime ideal $I_t$. 
If one can show that $\mathfrak{D}$ belongs to the symbolic power $I_t^{(h)}$, where $h=\height(I_t)$, then it would follow
\[\mathfrak{D}^{p-1}I_tS_{I_t}\subseteq \left((I_tS_{I_t})^{(h)}\right)^{p-1}I_tS_{I_t} = (I_tS_{I_t})^{h(p-1)+1}\subseteq I_t^{[p]}S_{I_t},\] 
where $(I_tS_{I_t})^{(h)}=(I_tS_{I_t})^{h}$, because $I_tS_{I_t}$ is the (unique) maximal ideal of $S_{I_t}$,
and the last containment follows from Lemma \ref{lem: sym power in colon ideal}. We now show $\mathfrak{D}\in I_t^{(h)}$. In \cite[Theorem 1]{MR0352082}, Kutz has shown that $\dim R = n(t-1)-\frac{1}{2}(t-1)(t-2)$, hence it follows
\[h=\dim S-\dim R = \frac{n(n+1)}{2}-n(t-1)+\frac{(t-1)(t-2)}{2}.\]
In general, suppose $\Delta$ is a product $\delta_1\cdots\delta_u$ of $a_i$-minors $\delta_i$ of $X$, where the vector $(a_1,\dots,a_u)\in\N^u$ is called the \emph{shape} of $\Delta$. With the previous notation, the symbolic power $I_t^{(k)}$ is described by
\[\left[\Delta\in I_t^{(k)} \right] \ \Longleftrightarrow \ \left[ \text{$a_i\le n \enspace \forall i$ and}\enspace \gamma_t(\Delta):=\sum_{i=1}^u\max\{0, a_i-t+1\}\ge k \right]\]
(see \cite[Proposition 4.3]{MR3335294}.) We must show $\gamma_t(\mathfrak{D})\ge h$. Observe that the product $\mathfrak{D}=d_1d_2\cdots d_n$ has shape $(1,2,\dots,n)$, and note that $i-t+1>0 \Leftrightarrow i>t-1$, hence $\max\{0, i-t+1\}=i-t+1$ for $t\le i\le n$. Thus, one has
\begin{align*}
    \gamma_t(\mathfrak{D}) &= \sum_{i=t}^n(i-t+1) \\
    &= \sum_{i=t}^ni + \sum_{i=t}^n(1-t) \\
    &= \frac{n(n+1)}{2} - \frac{t(t-1)}{2} + (n-t+1)(1-t) \\
    &= \frac{n(n+1)}{2}-n(t-1)+\frac{(t-1)(t-2)}{2},
\end{align*}
i.e., $\gamma_t(\mathfrak{D})=\height(I_t)=h$, so it follows that $\mathfrak{D}\in I_t^{(h)}$. 
\end{proof}

\begin{example}\label{ex: 3x3 symm matrix}
For the $3\times3$ symmetric matrix
\[X = \begin{pmatrix}
    x_{11} & x_{12} & x_{13} \\
    x_{12} & x_{22} & x_{23} \\
    x_{13} & x_{23} & x_{33} 
\end{pmatrix}\]
the $i$-minors $d_i$ ($i=1,2,3$) of the $i\times i$ botom-left corner of $X$ are 
\[d_1=x_{13}, \quad d_2=\begin{vmatrix} x_{12} & x_{22} \\ x_{13} & x_{23} \end{vmatrix}, \quad d_3=|X|,\]
which have leading terms $\mathrm{in}(d_1)=x_{13},\ \mathrm{in}(d_2)=x_{12}x_{23}$, and $\mathrm{in}(d_3)=x_{11}x_{22}x_{33}$, hence $\mathfrak{D}=d_1d_2d_3$ has squarefree leading term $x_{11}x_{12}x_{13}x_{22}x_{23}x_{33}$.
\end{example}

\begin{lemma}\cite[Lemma 1.71]{Eloisa}\label{lem: primary ideal}
Let $Q$ be a primary ideal of a Noetherian ring $R$. If $P$ is a prime ideal that contains $Q$, then $QR_P\cap R=Q$. 
\end{lemma}

\begin{proposition}\label{prop: second symbolic power}
Over a field $\Bbbk$ of arbitrary characteristic, let $\mathfrak{p}$ be the prime ideal of $R=\Bbbk[X]/I_t(X)$ generated by the $t-1$ minors of the first $t-1$ rows of $X$. Then the second symbolic power $\mathfrak{p}^{(2)}$ is principally generated by the $(t-1)$-minor in the top left-hand corner of $X$.
\end{proposition}

\begin{proof}
Let $d$ represent the $(t-1)$-minor in the top left corner of $X$. By Proposition (1) of \cite{MR0435063} the radical of the principal ideal $(d)=dR$ is the prime ideal $\mathfrak{p}$. To see that $(d)$ is primary note that $R$ is a Cohen-Macaulay domain by \cite[Theorem 1]{MR0352082}, and $d\in R$ is a nonzero divisor, hence $R/(d)$ is Cohen-Macaulay, so all associated primes of $(d)$ are minimal. Since $\sqrt{(d)}=\mathfrak{p}$ is prime, it follows that $(d)$ is $\mathfrak{p}$-primary. 

By Proposition (2) of \cite{MR0435063}, the local ring $R_{\mathfrak{p}}$ is a discrete valuation ring with valuation $v_{\mathfrak{p}}$, and $v_{\mathfrak{p}}(d)=2$. Let $(d)R_{\mathfrak{p}}$ be the extension of $(d)$ in $R$ to $R_{\mathfrak{p}}$. This is a nonzero proper ideal in $R_{\mathfrak{p}}$ by the valuation of $d$. Since every proper nontrivial ideal of the DVR $R_{\mathfrak{p}}$ is a power of the unique maximal ideal $\mathfrak{p}R_{\mathfrak{p}}$, we have $(d)R_{\mathfrak{p}}=\mathfrak{p}^mR_{\mathfrak{p}}=\{a\in R_{\mathfrak{p}}\mid v_{\mathfrak{p}}(a)\ge m\}$ for some $m\in\N$. From $v_{\mathfrak{p}}(d)=2$, it follows that $(d)R_{\mathfrak{p}}\subseteq\mathfrak{p}^2R_{\mathfrak{p}}$. Also, $d\in\mathfrak{p}^mR_{\mathfrak{p}}$ implies $2=v_{\mathfrak{p}}(d)\ge m$, so $m$ is either 1 or 2. However, $\mathfrak{p}^mR_{\mathfrak{p}}\subseteq\mathfrak{p}^2R_{\mathfrak{p}}$ can only mean $m=2$, so it follows $(d)R_{\mathfrak{p}}=\mathfrak{p}^2R_{\mathfrak{p}}$.

Since $d\in\mathfrak{p}$, one has $(d)\subseteq\mathfrak{p}$, and it follows from Lemma \ref{lem: primary ideal} and our previous discussion that $(d)=(d)R_{\mathfrak{p}}\cap R=\mathfrak{p}^2R_{\mathfrak{p}}\cap R = \mathfrak{p}^{(2)}$.
\end{proof}

\begin{theorem}\label{thm: fpt symm det ring}
The $F$-pure threshold of the symmetric determinantial ring $R=\Bbbk[X]/I_t(X)$ over an $F$-finite field $\Bbbk$ is 
\[ \fpt(\mathfrak{m}) = \frac{n(t-1)}{2}.\]
When $\Bbbk$ has characteristic zero, the log canonical threshold of $\mathfrak{m}$ shares the same value.
\end{theorem}

\begin{proof}
The ring $R=\Bbbk[X]/I_t$ is Gorenstein if and only if $n\equiv t \Mod 2$, hence it follows from \cite[Theorem 5.2]{MR3778235} and equation (\ref{equ: symm det a-invariant}) that $\fpt(\mathfrak{m})=-a(R)=(t-1)\frac{n}{2}$, when $n\equiv t \Mod 2$. Now, suppose $R$ is non-Gorenstein, i.e., $n\not\equiv t \Mod 2$. Its $a$-invariant and graded canonical module are respectively $a(R)=-(t-1)\frac{n+1}{2}$ and $\omega_R=\mathfrak{p}\left(-(t-1)\frac{n-1}{2}\right)$ (see (\ref{equ: symm det graded canonical mod})). Since the class group $\cl(R)$ is cyclic of order $c=2$, the ring $R$ is $\Q$-Gorenstein, so we apply \cite[Proposition 4.5]{MR3719471}. We determine the generating degree $D$ of $\omega_R^{(2)}$. First, note that $\mathfrak{p}^{(2)}$ is principally generated by a $(t-1)$-minor, by Proposition \ref{prop: second symbolic power}, hence as graded modules $\mathfrak{p}^{(2)}\cong R(-(t-1))$. The second symbolic power of $\omega_R$ is then
\begin{align*}
    \omega_R^{(2)} = \left(\mathfrak{p}\left(\frac{-(t-1)(n-1)}{2}\right)\right)^{(2)} 
    &= \mathfrak{p}^{(2)}\left(2\cdot\frac{-(t-1)(n-1)}{2}\right) \\
    &= R\big(-(t-1)\big)\big(-(t-1)(n-1)\big) \\
    &= R\big(-(t-1)-(t-1)(n-1)\big) \\
    &= R\left(-(t-1)n\right),
\end{align*}
which is generated in degree $D=(t-1)n$. It follows that $\fpt(\mathfrak{m})=D/c=n(t-1)/2$. 

In characteristic zero, one has by \cite[Theorem 3.4]{MR2185754} that the log canonical threshold of $R$ is 
\[\lct(\mathfrak{m})=\lim_{p\to\infty}\fpt(\mathfrak{m}_{R_p})=\frac{n(t-1)}{2},\]
where $\mathfrak{m}_{R_p}$ is the image of $\mathfrak{m}$ in $R_p$, the prime reduction modulo $p$ of $R$.
\end{proof}

$\,$

$\,$

\bibliographystyle{alpha}
\bibliography{bibliography.bib}

\begin{thebibliography}{DSMnNnB23}

\bibitem[Avr79]{MR0570242}
Luchezar~L. Avramov.
\newblock A class of factorial domains.
\newblock {\em Serdica}, 5(4):378--379, 1979.

\bibitem[Bar94]{MR1255876}
Margherita Barile.
\newblock The {C}ohen-{M}acaulayness and the {$a$}-invariant of an algebra with straightening laws on a doset.
\newblock {\em Comm. Algebra}, 22(2):413--430, 1994.

\bibitem[Bata01]{MR1888527}
C.~B\u ae\c~tic\u a.
\newblock {$F$}-rationality of algebras defined by {P}faffians.
\newblock volume 3(53), pages 139--144. 2001.
\newblock Memorial issue dedicated to Nicolae Radu.

\bibitem[BH92]{MR1188581}
Winfried Bruns and J\"{u}rgen Herzog.
\newblock On the computation of {$a$}-invariants.
\newblock {\em Manuscripta Math.}, 77(2-3):201--213, 1992.

\bibitem[CMSV18]{MR3836659}
Aldo Conca, Maral Mostafazadehfard, Anurag~K. Singh, and Matteo Varbaro.
\newblock Hankel determinantal rings have rational singularities.
\newblock {\em Adv. Math.}, 335:111--129, 2018.

\bibitem[Con94]{MR1309380}
Aldo Conca.
\newblock Symmetric ladders.
\newblock {\em Nagoya Math. J.}, 136:35--56, 1994.

\bibitem[CSV24]{CSV}
Aldo Conca, Anurag~K. Singh, and Matteo Varbaro.
\newblock Invariant rings of the special orthogonal group have nonunimodal ${h}$-vectors.
\newblock preprint arXiv:2406.14439, 2024.

\bibitem[DN01]{MR1902527}
Emanuela De~Negri.
\newblock Some results on {H}ilbert series and {$a$}-invariant of {P}faffian ideals.
\newblock {\em Math. J. Toyama Univ.}, 24:93--106, 2001.

\bibitem[DSMnNnB23]{MR4603914}
Alessandro De~Stefani, Jonathan Monta\~{n}o, and Luis N\'{u}\~{n}ez Betancourt.
\newblock Frobenius methods in combinatorics.
\newblock {\em S\~{a}o Paulo J. Math. Sci.}, 17(1):387--429, 2023.

\bibitem[DSNnB18]{MR3778235}
Alessandro De~Stefani and Luis N\'{u}\~{n}ez Betancourt.
\newblock {$F$}-thresholds of graded rings.
\newblock {\em Nagoya Math. J.}, 229:141--168, 2018.

\bibitem[Fed83]{MR701505}
Richard Fedder.
\newblock {$F$}-purity and rational singularity.
\newblock {\em Trans. Amer. Math. Soc.}, 278(2):461--480, 1983.

\bibitem[Got77]{MR0435063}
Shiro Goto.
\newblock The divisor class group of a certain {K}rull domain.
\newblock {\em J. Math. Kyoto Univ.}, 17(1):47--50, 1977.

\bibitem[Got79]{MR0545715}
Shiro Goto.
\newblock On the {G}orensteinness of determinantal loci.
\newblock {\em J. Math. Kyoto Univ.}, 19(2):371--374, 1979.

\bibitem[Gri]{Eloisa}
Elo\`{i}sa Grifo.
\newblock Symbolic powers.
\newblock \url{https://eloisagrifo.github.io/SymbolicPowers.pdf}.
\newblock [Online; accessed 17-July-2024].

\bibitem[JMnV15]{MR3335294}
Jack Jeffries, Jonathan Monta\~{n}o, and Matteo Varbaro.
\newblock Multiplicities of classical varieties.
\newblock {\em Proc. Lond. Math. Soc. (3)}, 110(4):1033--1055, 2015.

\bibitem[Kut74]{MR0352082}
Ronald~E. Kutz.
\newblock Cohen-{M}acaulay rings and ideal theory in rings of invariants of algebraic groups.
\newblock {\em Trans. Amer. Math. Soc.}, 194:115--129, 1974.

\bibitem[MTW05]{MR2185754}
Mircea Musta\c{t}\v{a}, Shunsuke Takagi, and Kei-ichi Watanabe.
\newblock F-thresholds and {B}ernstein-{S}ato polynomials.
\newblock In {\em European {C}ongress of {M}athematics}, pages 341--364. Eur. Math. Soc., Z\"{u}rich, 2005.

\bibitem[STV17]{MR3719471}
Anurag~K. Singh, Shunsuke Takagi, and Matteo Varbaro.
\newblock A {G}orenstein criterion for strongly {$F$}-regular and log terminal singularities.
\newblock {\em Int. Math. Res. Not. IMRN}, (21):6484--6522, 2017.

\bibitem[TW04]{MR2097584}
Shunsuke Takagi and Kei-ichi Watanabe.
\newblock On {F}-pure thresholds.
\newblock {\em J. Algebra}, 282(1):278--297, 2004.

\end{thebibliography}

\end{document}